\documentclass[10pt,reqno]{amsart}

\usepackage{amsmath,amssymb,amsfonts,latexsym,amsthm}
\usepackage{mathrsfs}
\usepackage{hyperref}
\numberwithin{equation}{section}

\usepackage{graphicx}

\usepackage{ifthen} 
\calclayout
\provideboolean{shownotes} 
\setboolean{shownotes}{false} 
\newcommand{\margnote}[1]{
	\ifthenelse{\boolean{shownotes}}%
	{\marginpar{\raggedright\tiny\texttt{#1}}}%
	{}%
}

\newcommand{\hole}[1]{
	\ifthenelse{\boolean{shownotes}}%
	{\begin{center} \fbox{ \rule {.25cm}{0cm}
				\rule[-.1cm]{0cm}{.4cm} \parbox{.85\textwidth}{\begin{center}
						\texttt{#1}\end{center}} \rule {.25cm}{0cm}}\end{center}}
	{}
}

\theoremstyle{plain}
\newtheorem{defi}{Definition}[section]
\newtheorem{theo}{Theorem}[section]
\newtheorem{lema}{Lemma}

\newtheorem{res}{Result}

\newtheorem{coro}{Corollary}
\newtheorem{rem}{Remark}
\theoremstyle{remark}

\theoremstyle{remark}

\begin{document}

	\title[]{On the equations of compressible fluid dynamics with Cattaneo-type extensions for the heat flux: Symmetrizability and relaxation structure}
	
	\author[F. Angeles]{Felipe Angeles}
	
	\address{{\rm (F. Angeles)} IIMAS\\Universidad Nacional Aut\'{o}noma 
		de M\'{e}xico\\Circuito Exterior s/n, Ciudad de M\'{e}xico C.P. 04510 (Mexico)}
	
	\email{teojkd@ciencias.unam.mx}

	\begin{abstract}	
	The aim of this work is twofold. From a mathematical point of view, we show the existence of a hyperbolic system of equations that is not symmetrizable in the sense of Friedrichs. Such system appears in the theory of compressible fluid dynamics with Cattaneo-type extensions for the heat flux. In contrast, the linearizations of such system around constant equilibrium solutions have Friedrichs symmetrizers. Then, from a physical perspective, we aim to understand the relaxation term appearing in this system. By noticing the violation of the Kawashima-Shizuta condition, locally and smoothly, with respect to the Fourier frequencies, we construct persistent waves, i.e., solutions preserving the $L^{2}$ norm for all times that are not dissipated by the relaxation terms.
	\end{abstract}
	\maketitle
	
	\setcounter{tocdepth}{1}

	\section{Introduction}
Consider the local form of conservation of mass ($\rho$), the balance of linear momentum ($\rho v$) and the balance of total energy ($E=\tfrac{1}{2}\rho|v|^{2}+e$) for a compressible inviscid fluid in space, that is, 
	\begin{eqnarray}
	\rho_{t}+\nabla\cdot(\rho v)&=&0,\label{eq:1}\\
	(\rho v)_{t}+\nabla\cdot(\rho v\otimes v)&=&-\nabla\cdot p,\label{eq:2}\\
	(\rho E)_{t}+\nabla\cdot(\rho E v)&=&-\nabla\cdot(pv)-\nabla\cdot q,\label{eq:3}
\end{eqnarray}
together with the following evolution equation for the heat flux ($q$), 
	\begin{equation}
	\tau\left[\partial_{t}q+v\cdot\nabla q-\frac{1}{2}\left(\nabla v-\nabla v^{\top}\right)q+\frac{\lambda}{2}\left(\nabla v+\nabla v^{\top}\right) q+\nu(\nabla\cdot v)q\right]+q=-\kappa\nabla\theta,
	\label{eq:objder}
\end{equation}
where $\theta$ is the temperature field, $\kappa$ is the thermal conductivity and $\tau$ is the thermal relaxation time. In \cite{morro2}, Morro shows that for any pair of invariant scalars $(\lambda,\nu)$, the term between brackets in \eqref{eq:objder}, is an objective derivative. It is used in susbstitution of the partial time derivative of $q$, in order to obtain a frame invariant formulation for the Maxwell-Cattaneo heat flux equation, namely,
\begin{align*}
\tau\partial_{t}q+q=-\kappa\nabla\theta.
\end{align*}
For this reason, we refer to \eqref{eq:objder} as an \emph{objective Cattaneo's type equation for the heat flux}. Notice that for no values of $(\lambda,\nu)$, the evolution equation \eqref{eq:objder} yields the material formulation of the Maxwell-Cattaneo law proposed by Christov and Jordan \cite{jordy}, namely,
\begin{align}
	\tau\left[q_{t}+v\cdot\nabla q\right]+q=-\kappa\nabla\theta.\label{eq:materialflux}
\end{align}
Although \eqref{eq:materialflux} is a Galilean invariant formulation of the Maxwell-Cattaneo law \cite{jordy}, Morro points out that it is not objective \cite{morro2}. Nonetheless, many of the well known objective formulations for the heat flux can be obtained as particular cases  of $(\lambda,\nu)$ in \eqref{eq:objder} (see the discussion of Morro in \cite{amorro}, \cite{morro2} and the references therein). It is worth mentioning the case $(\lambda,\nu)=(-1,1)$, proposed by Christov, which corresponds to the Lie-Oldroyd upper convected derivative \cite{christov}. This particular formulation has the propery that, when it is combined with the material invariant form of the balance of internal energy, it allows for the heat flux to be eliminated in order to obtain a single hyperbolic equation for the temperature. Then, Straughan proposes the coupling between \eqref{eq:1}-\eqref{eq:3} and Christov's evolution equation for the heat flux (i.e. \eqref{eq:objder} with $(\lambda,\nu)=(-1,1)$) as a model for an inviscid fluid \cite{stra}. However, the quasilinear form of this coupling, known as the Cattaneo-Christov system of equations, it is not \emph{hyperbolic} \cite{angelesnon}. Furthermore, when $(\lambda,\nu)$ are considered real valued functions of $(\rho,\theta)$, the quasilinear form of the coupling between \eqref{eq:1}-\eqref{eq:3} and \eqref{eq:objder} is a \emph{hyperbolic system of equations} if and only if $(\lambda(\rho,\theta),\nu(\rho,\theta))=(1,-1)$ for every $\rho>0$ and $\theta>0$. Throughout this paper we refer to this hyperbolic system as the $(1,-1)$-quasilinear system of equations.\
Consider the first order partial differential equation
\begin{align}
	\label{eq:stronghyp}
	U_{t}+\sum_{j=1}^{d}\widetilde{A}^{j}(x,t)\partial_{j}U+\widetilde{D}(x,t)U&=0
\end{align}
where $\widetilde{A}^{j}(x,t)$ and $\widetilde{D}(x,t)$ are square matrices of order $n$, infinitely differentiable with respect to $x\in\mathbb{R}^{d}$ and $t>0$. The Cauchy problem for \eqref{eq:stronghyp} is $L^{2}$ well-posed if for any $U_{0}\in L^{2}(\mathbb{R}^{d})$ given at $t=0$, there is a unique solution $U=U(x,t)$ in $\mathcal{C}\left([0,T];L^{2}(\mathbb{R}^{3})\right)$ such that
\begin{align}
	\|U(\cdot,t)\|_{L^{2}(\mathbb{R}^{d})}\leq C\|U_{0}\|_{L^{2}(\mathbb{R}^{d})}\quad\mbox{for all}\quad t\in[0,T].\label{eq:l2estimate}
\end{align}
We say that \eqref{eq:stronghyp} is a hyperbolic system of equations if $\widetilde{A}(x,t;\xi):=\sum_{j}\widetilde{A}^{j}(x,t)\xi_{j}$ is a real diagonalizable matrix for any $(x,t;\xi)\in\mathbb{R}^{d}\times[0,T]\times\mathbb{S}^{d-1}$. In such case, its eigenvalues are called \emph{characteristic speeds}. Although the hyperbolicity of a first order linear or a quasilinear system of equations is a necessary property for the $L^{2}$ well-posedness of its Cauchy problem (cf. \cite{kajineed}, \cite{kajitani}, \cite{kano}, \cite{kasahara}, \cite{gstran}), it is not sufficient \cite{gstran}.\\
Another necessary condition has been provided by Ivrii and Petkov. They have shown that if estimate  \eqref{eq:l2estimate} is valid for $u\in\mathcal{C}_{0}\left(\mathbb{R}^{d}\times(0,T)\right)$, then there exists a bounded \emph{microlocal symmetrizer}, that is, a bounded symmetric matrix $S(x,t,\xi)$, uniformly positive definite, homogeneous of degree $0$ in $\xi\neq0$ and such that $S(x,t;\xi)\widetilde{A}(x,t;\xi)$ is symmetric (cf. \cite{metivierl2}). Moreover, by Kreiss Matrix Theorem \cite[Theorem 2.3.2]{otto}, the $L^{2}$ well-posedness of \eqref{eq:stronghyp} with frozen coefficients at $(x,t)$, is equivalent to the existence of a positive definite Hermitian matrix $H(x,t;\xi)$ such that $H(x,t;\xi)\left(-i\widetilde{A}(x,t;\xi)\right)=i\widetilde{A}^{\top}(x,t;\xi)H(x,t;\xi)$. However, outside the case of systems with constant coefficients, the existence of bounded \emph{symmetrizers} is insufficient for the existence of the $L^{2}$-energy estimate \eqref{eq:l2estimate}. We can find two counterexamples, one in \cite{gstran} and another in \cite[Section 3]{metivierl2}. The Lipschitz continuity of the symmetrizer $S$, with respect to $(x,t;\xi)$ for $\xi\neq 0$, is sufficient in order to deduce the existence of $L^{2}$ energy estimates. Moreover, the condition of Lipschitz regularity is sharp in the sense that, we can find Cauchy problems with non-Lipschitz symmetrizers that are ill-posed in $L^{2}$ as well as in $\mathcal{C}^{\infty}$ (see, \cite[Section 3.3]{metivierl2}).\\
When the symmetrizer is independent of $\xi$, we say that it is a \emph{Friedrichs symmetrizer} (see \cite[Definition 2.1]{benzoserre}). In this case, \eqref{eq:l2estimate} is easily obtained by the standard \emph{energy method}.\\
Assume that \eqref{eq:stronghyp} is of quasilinear type, that is, $\widetilde{A}^{j}(x,t)=A^{j}(U(x,t))$ and $\widetilde{D}(x,t)=D(U(x,t))$ where the mappings $U\mapsto A^{j}(U)$, $U\mapsto D(U)$ are smooth and that the system can be derived from a conservative form (i.e. the coefficients $A^{j}$ are Jacobian matrices of flux functions). Then, a Friedrichs symmetrizer $S=S(U)$ (see, \cite[Definition 10.1]{benzoserre}) can be obtained as the Hessian matrix of a convex entropy function (cf. \cite{daf,serre}). Many of the standard models in compressible fluid dynamics have such structure (see \cite{daf}). Nonetheless, one can also look for a  Friedrichs symmetrizer by the method of inspection. This is the case of the one dimensional version of the Cattaneo-Christov system \cite[Section 3]{amp} and of the Cattaneo-Christov-Jordan system in several space dimensions, the latter being the coupling between \eqref{eq:1}-\eqref{eq:3} and \eqref{eq:materialflux} (see, \cite[Section 4]{angelesnon}). None of these systems has a conservative structure. Kawashima and Yong showed that it is possible to define a convex entropy for quasilinear systems not necessarily in conservative form \cite{kawayong}: if the symmetrizer is the Jacobian of a diffeomorphic change of variables, $S(U)=D\Psi(U)$, then a convex entropy function can be introduced. However, it is easily seen that the symmetrizer of the Cattaneo-Christov-Jordan system, in one and several space dimensions, does not fulfill such property and the systems are not necessarily endowed with convex entropy functions. This raises the question:
\begin{itemize}
	\item [(Q1)] Can we find a Friedrichs symmetrizer for the $(1,-1)$ quasilinear system of equations?
\end{itemize}
Surprisingly, the answer to this question is negative (see Theorem \ref{Fsymm}). Therefore, the $(1,-1)$-quasilinear system of equations is a hyperbolic system without a Friedrichs symmetrizer. To the author's knowledge, this property has never been seen before for multidimensional systems appearing from physical considerations. In this regard, it is important to mention the example provided by Lax in \cite{laxsymm}. In this work, Lax exhibits a hyperbolic system of equations ($d=2$ and $n=3$) with constant coefficients matrices that cannot be made symmetric by the same similarity transformation.\\
Even when the $(1,-1)$-quasilinear system is not Friedrichs symmetrizable, we argue that a microlocal symmetrizer can be found, at least at the level of linearization. Furthermore, the linearization of this system around \emph{constant equilibrium states}, does have a Friedrichs symmetrizer. In consequence, the \emph{strict dissipativity} of the system is equivalent to its \emph{genuinely coupling} \cite{kawazuta}. In lay terms, the strict dissipativity is the property that solutions of the linearized systems around constant equilibrium states show some decay structure. On the other hand, the genuinely coupling condition for a system of the form \eqref{eq:stronghyp} with constant coefficients, states that, the dissipation terms do not allow solution of traveling wave type to be, simultaneously, solutions to the hyperbolic system without dissipation. In the case of every $(\lambda,\nu)$-quasilinear system, the thermal relaxation accounts for dissipation effects.\\
Observe that, the one dimensional version of the $(1,-1)$-quasilinear system coincides with the one dimensional version of the inviscid Cattaneo-Christov system. In \cite{amp}, it was shown that, the linearizations around constant equilibrium states of the one dimensional Cattaneo-Christov system are strictly dissipative. In consequence, there are \emph{compensation matrices} that allow to compute energy decay rates, see \cite[Lemma 5.1 and Theorem 5.2]{amp}. This means that, the presence of the relaxation term in the linearization of the one dimensional $(1,-1)$-quasilinear system produces the decay of the $L^{2}$ norm of its solutions as $t\rightarrow\infty$.
 However, in Theorem \ref{openkawa} we show that, in several space dimensions, the system is not genuinely coupled (in the sense of Kawashima and Shizuta \cite{kawazuta}) and the strict dissipativity does not hold true.\\
 The violation of the Kawashima-Shzita condition is related with the possibility of finding \emph{non-constant equilibrium solutions}. As Mascia and Natalini argue \cite{mascialini}, the set of equilibrium solutions or \emph{Maxwellians} can be much richer if the Kawashima-Shizuta condition is not satisfied. In the one dimensional strictly hyperbolic case, they showed that, the linearization around equilibirum solutions of hyperbolic balance laws with a relaxation source, have non-consant Maxwellians if and only if the Kawashima-Shizuta condition is violated \cite[Proposition 1]{mascialini}.\\
 In this work we show the existence of \emph{persistent waves} for the  linearization around constant equilibrium states of the $(1,-1)$-quasilinear system. That is, we construct non-constant Maxwellians with constant $L^{2}$ norm for all times. Thus showing the existence of solutions that escape the effect of the relaxation terms.

\section{Thermodynamical assumptions}	
	Throughout this paper we make the following thermodynamical assumptions:
	\begin{itemize}
		\item [(\textbf{C1})] The independent thermodynamical variables are the mass density field $\rho(x,t)>0$ and the absolute temperature field $\theta(x,t)>0$. They vary within the domain $\mathcal{D}=\{(\rho,\theta)\in\mathbb{R}^{2}~|~\rho>0,\theta>0\}$. The thermodynamic pressure $p$, the internal energy $e$ and the thermal conductivity $\kappa$ are given smooth functions of $(\rho,\theta)\in\mathcal{D}$.
		\item [(\textbf{C2})] The fluid satisfies the following conditions $p,p_{\rho},p_{\theta},e_{\theta},\kappa>0$ for all $(\rho,\theta)\in\mathcal{D}$.
	\end{itemize}
In particular, assumption (\textbf{C2}) refers to compressible fluids satisfying the standard assumptions of Weyl \cite{weyl}.
	\section{A hyperbolic system without a Friedrichs simmetrizer}
We begin this section by writing in quasilinear form of the coupling between the following Cattaneo-type extension for the heat flux,
	\begin{equation}
	\tau\left[\partial_{t}q+v\cdot\nabla q+\nabla v^{\top}q-(\nabla\cdot v)q\right]+q=-\kappa\nabla\theta,
	\label{eq:hypextension}
\end{equation}
and equations \eqref{eq:1}-\eqref{eq:3} in three space dimensions. Notice that \eqref{eq:hypextension} was obtained by setting $(\lambda,\nu)=(1,-1)$ in \eqref{eq:objder}.\\
Let the state variable be $U=(\rho,v,\theta,q)^{\top}\in\mathcal{O}\subset\mathbb{R}^{8}$, where $\mathcal{O}:=\left\lbrace(\rho,v,\theta,q)\in\mathbb{R}^{8}:\rho>0,\theta>0\right\rbrace$. The quasilinear form between \eqref{eq:1}-\eqref{eq:3} and \eqref{eq:hypextension} is 
\begin{align}
	U_{t}+A^{j}(U)\partial_{j}U+Q(U)=0,\label{eq:quasi2}
\end{align}
	where repeated index notation has been used in the space derivatives $\partial_{j}$ and $i=1,2,3$. Here, once $U\in\mathcal{O}$ is given, each coefficient $A^{j}(U)$ is a matrix of order $8\times 8$ and $Q(U)$ is a vector in $\mathbb{R}^{8}$. For $\xi=(\xi_{1},\xi_{2},\xi_{3})\in\mathbb{R}^{3}\setminus\{0\}$ and $U\in\mathcal{O}$, we define the symbol
	\small
	\begin{align}
		\label{eq:Axi}
		A(\xi;U):=\sum_{i=1}^{3}A^{i}(U)\xi_{i}=\left(\begin{array}{cccccccc}
			\xi\cdot v&\xi_{1}\rho&\xi_{2}\rho&\xi_{3}\rho&0&0&0&0\\
			\xi_{1}\frac{p_{\rho}}{\rho}&\xi\cdot v&0&0&\xi_{1}\frac{p_{\theta}}{\rho}&0&0&0\\
			\xi_{2}\frac{p_{\rho}}{\rho}&0&\xi\cdot v&0&\xi_{2}\frac{p_{\theta}}{\rho}&0&0&0\\
			\xi_{3}\frac{p_{\rho}}{\rho}&0&0&\xi\cdot v&\xi_{3}\frac{p_{\theta}}{\rho}&0&0&0\\
			0&\xi_{1}\frac{\theta p_{\theta}}{\rho e_{\theta}}&\xi_{2}\frac{\theta p_{\theta}}{\rho e_{\theta}}&\xi_{3}\frac{\theta p_{\theta}}{\rho e_{\theta}}&v\cdot \xi&\frac{\xi_{1}}{\rho e_{\theta}}&\frac{\xi_{2}}{\rho e_{\theta}}&\frac{\xi_{3}}{\rho e_{\theta}}\\
			0&&&&\frac{\xi_{1}\kappa}{\tau}&v\cdot\xi&0&0\\
			0&&\mathcal{Q}_{1,-1}(\xi;q)&&\frac{\xi_{2}\kappa}{\tau}&0&v\cdot\xi&0\\
			0&&&&\frac{\xi_{3}\kappa}{\tau}&0&0&v\cdot\xi\\
		\end{array}\right),
	\end{align}
	\normalsize
	where, for each $\xi\in\mathbb{R}^{3}\setminus\{0\}$, the sub-block matrix $\mathcal{Q}_{1,-1}(q;\xi)$ is of order $3\times 3$ and given as
	\small
	\begin{equation}
		\mathcal{Q}_{1,-1}(q;\xi)=\left(\begin{array}{ccc}
			0&\xi_{1}q_{2}-\xi_{2}q_{1}&\xi_{1}q_{3}-\xi_{3}q_{1}\\
			\xi_{2}q_{1}-\xi_{1}q_{2}&0&\xi_{2}q_{3}-\xi_{3}q_{2}\\
			\xi_{3}q_{1}-\xi_{1}q_{3}&\xi_{3}q_{2}-\xi_{2}q_{3}&0
		\end{array}\right),
		\label{eq:23}
	\end{equation}
	\normalsize
and
\begin{align}
	Q(U)=\frac{1}{\tau}\left(0,0,0,0,0,q_{1},q_{2},q_{3}\right)^{\top}.\label{eq:QU2}
\end{align}
Let us recall the definition of Friedrichs symmetrizability for a quasilinear system of equations (cf. \cite[Definition 10.1]{benzoserre}).
\begin{defi}
The quasilinear system \eqref{eq:quasi2} is called Friedrichs-symmetrizable in $\mathcal{O}$ if there exists a smooth mapping $S:\mathcal{O}\rightarrow\operatorname{Symm}_{8}(\mathbb{R})$ (real symmetric matrices of eight order) such that $S(U)$ is positive definite, and the matrices $S(U)A^{j}(U)$ are symmetric for all $U\in\mathcal{O}$. Then, we refer to $S=S(U)$ as the Friedrichs symmetrizer for \eqref{eq:quasi2}.
\end{defi}
\begin{rem}
	\label{remlinearf}
Let $V\in\mathcal{O}$ be a smooth state. By a linearization of \eqref{eq:quasi2} around the state $V\in\mathcal{O}$ we mean the linear system 
\begin{align}
U_{t}+A^{j}(V)\partial_{j}U+DQ(V)U=0,\label{eq:linearf}
\end{align}
where $DQ(V)$ stands as the Jacobian matrix of \eqref{eq:QU2} evaluated at $V$. If $S=S(U)$ is a Friedrichs symmetrizer for \eqref{eq:quasi2} then, the matrix $S(V)$ is symmetric, positive definite and symmetrizes the symbol $A(\xi;V)$ independently of $\xi\in\mathbb{R}^{3}\setminus\{0\}$. Therefore, $S(V)$ is a Friedrichs symmetrizer for \eqref{eq:linearf}.
\end{rem}
\begin{theo}
\label{Fsymm}
Under the thermodynamical assumptions (\textbf{C1}) and (\textbf{C2}), the $(1,-1)$-quasilinear system of equations, is not Friedrichs symmetrizable.
\end{theo}
\begin{proof}
If there is a Friedrichs symmetrizer $S=S(U)$ for the system \eqref{eq:quasi2} then,
\begin{align}
S(U)A(\xi;U)&=A(\xi;U)^{\top}S(U)\quad\mbox{for all}\quad(\xi,U)\in\mathbb{S}^{2}\times\mathcal{O},\label{eq:symmprop}\\
S(U)&>0\quad\mbox{for all}\quad U\in\mathcal{O}.\label{eq:positivity}
\end{align}
Set any $\overline{U}=\left(\bar{\rho},\bar{v},\bar{\theta},\bar{q}_{1},\bar{q}_{2},\bar{q}_{3}\right)\in\mathcal{O}$ for which $\bar{q}_{i}\neq 0$ for all $i=1,2,3$ and let us denote by $s_{ij}$ any component of the matrix $S(\overline{U})$. If $\mathcal{S}_{k}(\overline{U})$ denotes the $kth$-row of $S(\overline{U})$ and $\mathcal{A}^{\ell}(\xi;\overline{U})$ denotes the $\ell th$-column of $A(\xi;\overline{U})$, condition \eqref{eq:symmprop} implies that 
\begin{align}
\mathcal{S}_{k}(\overline{U})\cdot \mathcal{A}^{\ell}(\xi;\overline{U})=\mathcal{S}_{\ell}(\overline{U})\cdot \mathcal{A}^{k}(\xi;\overline{U}),\quad\forall\quad k\neq \ell,\quad k,\ell=1,..,8,\quad\mbox{for all}\quad \xi\in\mathbb{S}^{2},\label{eq:symmetrizationeq}
\end{align}
where, $a\cdot b$ stands as the inner product between the vectors $a$ and $b$ in $\mathbb{R}^{8}$. Observe that \eqref{eq:symmetrizationeq} is made of $28$ equations. We refer to each of these equations as the $(k:\ell)$-equation. For simplicity, we will use the following notation for the components of $A(\xi;\overline{U})$ and $\mathcal{Q}_{1,-1}(\xi;\bar{q})$:
\begin{equation}
\label{eq:simplenotation}
\begin{aligned}
&\frac{p_{\rho}(\bar{\rho},\bar{\theta})}{\bar{\rho}}=:\alpha,\quad\frac{\bar{\theta}p_{\theta}(\bar{\rho},\bar{\theta})}{\bar{\rho}e_{\theta}(\bar{\rho},\bar{\theta})}=:\beta,\quad\frac{1}{\bar{\rho}e_{\theta}(\bar{\rho},\bar{\theta})}=\gamma,\quad\frac{\kappa(\bar{\rho},\bar{\theta})}{\tau}=:\delta,\quad\,\frac{p_{\theta}(\bar{\rho},\bar{\theta})}{\bar{\rho}}=:\eta\\
&\xi_{2}\bar{q}_{1}-\xi_{1}\bar{q}_{2}=:N_{12}(\xi;\bar{q}),\quad\xi_{3}\bar{q}_{1}-\xi_{1}\bar{q}_{3}=:N_{13}(\xi;\bar{q}),\quad\mbox{and}\quad \xi_{3}\bar{q}_{2}-\xi_{2}\bar{q}_{3}=:N_{23}(\xi;\bar{q}).
\end{aligned}
\end{equation}
By \eqref{eq:Axi} and \eqref{eq:simplenotation} we can decompose  \eqref{eq:symmetrizationeq} as the following set of equations
\begin{align}
&s_{11}\xi_{1}\bar{\rho}+s_{15}\xi_{1}\beta+s_{17}N_{12}(\xi;\bar{q})+s_{18}N_{13}(\xi;\bar{q})=\alpha\left(s_{22},s_{23},s_{24}\right)\cdot(\xi_{1},\xi_{2},\xi_{3}),\tag{1:2}\label{12}\\
&s_{11}\xi_{2}\bar{\rho}+s_{15}\xi_{2}\beta-s_{16}N_{12}(\xi;\bar{q})+s_{18}N_{23}(\xi;\bar{q})=\alpha\left(s_{23},s_{33},s_{34}\right)\cdot(\xi_{1},\xi_{2},\xi_{3}),\tag{1:3}\label{13}\\
&s_{11}\xi_{3}\bar{\rho}+s_{15}\xi_{3}\beta-s_{16}N_{13}(\xi;\bar{q})-s_{17}N_{23}(\xi;\bar{q})=\alpha\left(s_{24},s_{34},s_{44}\right)\cdot(\xi_{1},\xi_{2},\xi_{3}),\tag{1:4}\label{14}\\
&\left[\eta\left(s_{12},s_{13},s_{14}\right)+\delta\left(s_{16},s_{17},s_{18}\right)\right]\cdot\left(\xi_{1},\xi_{2},\xi_{3}\right)=\alpha\left(s_{25},s_{35},s_{45}\right)\cdot(\xi_{1},\xi_{2},\xi_{3}),\tag{1:5}\label{15}\\
&s_{15}\xi_{1}\gamma=\alpha\left(s_{26},s_{36},s_{46}\right)\cdot(\xi_{1},\xi_{2},\xi_{3}),\tag{1:6}\label{16}\\
&s_{15}\xi_{2}\gamma=\alpha\left(s_{27},s_{37},s_{47}\right)\cdot(\xi_{1},\xi_{2},\xi_{3}),\tag{1:7}\label{17}\\
&s_{15}\xi_{3}\gamma=\alpha\left(s_{28},s_{38},s_{48}\right)\cdot(\xi_{1},\xi_{2},\xi_{3}),\tag{1:8}\label{18}\\
&s_{12}\xi_{2}\bar{\rho}+s_{25}\xi_{2}\beta-s_{26}N_{12}(\xi;\bar{q})+s_{28}N_{23}(\xi;\bar{q})=s_{13}\xi_{1}\bar{\rho}+s_{35}\xi_{1}\beta+s_{37}N_{12}(\xi;\bar{q})+s_{38}N_{13}(\xi;\bar{q}),\tag{2:3}\label{23}\\
& s_{12}\xi_{3}\bar{\rho}+s_{25}\xi_{3}\beta-s_{26}N_{13}(\xi;\bar{q})-s_{27}N_{23}(\xi;\bar{q})=s_{14}\xi_{1}\bar{\rho}+s_{45}\xi_{1}\beta+s_{47}N_{12}(\xi;\bar{q})+s_{48}N_{13}(\xi;\bar{q}),\tag{2:4}\label{24}\\
&\left[\eta\left(s_{22},s_{23},s_{24}\right)+\delta\left(s_{26},s_{27},s_{28}\right)\right]\cdot(\xi_{1},\xi_{2},\xi_{3})=s_{15}\xi_{1}\bar{\rho}+s_{55}\xi_{1}\beta+s_{57}N_{12}(\xi;\bar{q})+s_{58}N_{13}(\xi;\bar{q}),\tag{2:5}\label{25}\\
&s_{25}\xi_{1}\gamma=s_{16}\xi_{1}\bar{\rho}+s_{56}\xi_{1}\beta+s_{67}N_{12}(\xi;\bar{q})+s_{68}N_{13}(\xi;\bar{q}),\tag{2:6}\label{26}\\
&s_{25}\xi_{2}\gamma=s_{17}\xi_{1}\bar{\rho}+s_{57}\xi_{1}\beta+s_{77}N_{12}(\xi;\bar{q})+s_{78}N_{13}(\xi;\bar{q}),\tag{2:7}\label{27}\\
&s_{25}\xi_{3}\gamma=s_{18}\xi_{1}\bar{\rho}+s_{58}\xi_{1}\beta+s_{78}N_{12}(\xi;\bar{q})+s_{88}N_{13}(\xi;\bar{q}),\tag{2:8}\label{28}\\
&s_{13}\xi_{3}\bar{\rho}+s_{35}\xi_{3}\beta-s_{36}N_{13}(\xi;\bar{q})-s_{37}N_{23}(\xi;\bar{q})=s_{14}\xi_{2}\bar{\rho}+s_{45}\xi_{2}\beta-s_{46}N_{12}(\xi;\bar{q})+s_{48}N_{23}(\xi;\bar{q}),\tag{3:4}\label{34}\\
&\left[\eta\left(s_{23},s_{33},s_{34}\right)+\delta\left(s_{36},s_{37},s_{38}\right)\right](\xi_{1},\xi_{2},\xi_{3})=s_{15}\xi_{2}\bar{\rho}+s_{55}\xi_{2}\beta-s_{56}N_{12}(\xi;\bar{q})+s_{58}N_{23}(\xi;\bar{q}),\tag{3:5}\label{35}\\
&s_{35}\xi_{1}\gamma=s_{16}\xi_{2}\bar{\rho}+s_{56}\xi_{2}\beta-s_{66}N_{12}(\xi;\bar{q})+s_{68}N_{23}(\xi;\bar{q})\tag{3:6}\label{36}\\
&s_{35}\xi_{2}\gamma=s_{17}\xi_{2}\bar{\rho}+s_{57}\xi_{2}\beta-s_{67}N_{12}(\xi;\bar{q})+s_{78}N_{23}(\xi;\bar{q}),\tag{3:7}\label{37}\\
&s_{35}\xi_{3}\gamma=s_{18}\xi_{2}\bar{\rho}+s_{58}\xi_{2}\beta-s_{68}N_{12}(\xi;\bar{q})+s_{88}N_{23}(\xi;\bar{q}),\tag{3:8}\label{38}\\
&\left[\eta\left(s_{24},s_{34},s_{44}\right)+\delta\left(s_{46},s_{47},s_{48}\right) \right]\cdot(\xi_{1},\xi_{2},\xi_{3})=s_{15}\xi_{3}\bar{\rho}+s_{55}\xi_{3}\beta-s_{56}N_{13}(\xi;\bar{q})-s_{57}N_{23}(\xi;\bar{q}),\tag{4:5}\label{45}\\
&s_{45}\xi_{1}\gamma=s_{16}\xi_{3}\bar{\rho}+s_{56}\xi_{3}\beta-s_{66}N_{13}(\xi;\bar{q})-s_{67}N_{23}(\xi;\bar{q}),\tag{4:6}\label{46}\\
&s_{45}\xi_{2}\gamma=s_{17}\xi_{3}\bar{\rho}+s_{57}\xi_{3}\beta-s_{67}N_{13}(\xi;\bar{q})-s_{77}N_{23}(\xi;\bar{q}),\tag{4:7}\label{47}\\
&s_{45}\xi_{3}\gamma=s_{18}\xi_{3}\bar{\rho}+s_{58}\xi_{3}\beta-s_{68}N_{13}(\xi;\bar{q})-s_{78}N_{23}(\xi;\bar{q}),\tag{4:8}\label{48}\\
&s_{55}\xi_{1}\gamma=\left[\eta\left(s_{26},s_{36},s_{46}\right)+\delta\left(s_{66},s_{67},s_{68}\right)\right]\cdot\left(\xi_{1},\xi_{2},\xi_{3}\right),\tag{5:6}\label{56}\\
&s_{55}\xi_{2}\gamma=\left[\eta\left(s_{27},s_{37},s_{47}\right)+\delta\left(s_{67},s_{77},s_{78}\right)\right]\cdot(\xi_{1},\xi_{2},\xi_{3}),\tag{5:7}\label{57}\\
&s_{55}\xi_{3}\gamma=\left[\eta\left(s_{28},s_{38},s_{48}\right)+\delta\left(s_{68},s_{78},s_{88}\right)\right]\cdot(\xi_{1},\xi_{2},\xi_{3}),\tag{5:8}\label{58}\\
&s_{56}\xi_{2}\gamma=s_{57}\xi_{1}\gamma,\tag{6:7}\label{67}\\
&s_{56}\xi_{3}\gamma=s_{58}\xi_{1}\gamma,\tag{6:8}\label{68}\\
&s_{57}\xi_{3}\gamma=s_{58}\xi_{2}\gamma.\tag{7:8}\label{78}
\end{align}
Assume that $\mathcal{G}=(G_{i})_{i=1}^{3}$ is a tridimensional vector with components independents of $\xi\in\mathbb{S}^{2}$. Then, it holds that
\begin{align}
	\mathcal{G}\cdot\xi=0\quad\forall~\xi\in\mathbb{S}^{2},\quad\mbox{implies}\quad G_{i}=0\quad\mbox{for all}\quad i=1,2,3.\label{eq:ortosphere}
\end{align} 
By applying \eqref{eq:ortosphere} to equations \eqref{16}, \eqref{17}, \eqref{18}, \eqref{67}, \eqref{68} and \eqref{78}, we obtain,
\begin{align}
s_{27}=s_{28}=s_{36}=s_{38}=s_{46}=s_{47}=s_{56}=s_{57}=s_{58}=0,\label{eq:firstbadge}
\end{align}
and $s_{15}\gamma=\alpha s_{26}=\alpha s_{37}=\alpha s_{48}$. Evaluating \eqref{46} and \eqref{48} in $\xi=\hat{e}_{2}$ and \eqref{36} in $\xi=\hat{e}_{3}$ yields the relations
\begin{align*}
-s_{67}N_{23}(\hat{e}_{2};\bar{q})=0,\quad-s_{78}N_{23}(\hat{e}_{2};\bar{q})=0,\quad s_{68}N_{23}(\hat{e}_{3};\bar{q})=0.
\end{align*}
Since $N_{23}(\hat{e}_{2},\bar{q})=-\bar{q}_{3}\neq 0$ and $N(\hat{e}_{3};\bar{q})=\bar{q}_{2}\neq0$ it follows that
\begin{align}
s_{67}=s_{68}=s_{78}=0.\label{eq:secondbadge}
\end{align}
By using \eqref{eq:firstbadge} and \eqref{eq:secondbadge} into equations \eqref{25}, \eqref{26}, \eqref{35}, \eqref{37}, \eqref{38}, \eqref{46}, \eqref{47} and \eqref{48} we obtain
\begin{align}
&\left[\eta\left(s_{22},s_{23},s_{24}\right)+\delta\left(s_{26},0,0\right)\right]\cdot(\xi_{1},\xi_{2},\xi_{3})=s_{15}\xi_{1}\bar{\rho}+s_{55}\xi_{1}\beta,\tag{2:5'}\label{25p}\\
&s_{25}\gamma=s_{16}\bar{\rho},\tag{2:6'}\label{26p}\\
&\left[\eta\left(s_{23},s_{33},s_{34}\right)+\delta\left(0,s_{37},0\right)\right]\cdot(\xi_{1},\xi_{2},\xi_{3})=s_{15}\xi_{2}\bar{\rho}+s_{55}\xi_{2}\beta,\tag{3:5'}\label{35p}\\
& s_{35}\gamma=s_{17}\bar{\rho},\tag{3:7'}\label{37p}\\
&s_{35}\xi_{3}\gamma=s_{18}\xi_{2}\bar{\rho}+s_{88}N_{23}(\xi;\bar{q}),\tag{3:8'}\label{38p}\\
&s_{45}\xi_{1}\gamma=s_{16}\xi_{3}\bar{\rho}-s_{66}N_{13}(\xi;\bar{q}),\tag{4:6'}\label{46p}\\
&s_{45}\xi_{2}\gamma=s_{17}\xi_{3}\bar{\rho}-s_{77}N_{23}(\xi;\bar{q}),\tag{4:7'}\label{47p}\\
&s_{45}\gamma=s_{18}\bar{\rho}\tag{4:8'}\label{48p}.
\end{align}
In particular, by \eqref{eq:ortosphere}, \eqref{25p} and \eqref{35p}, 
\begin{align*}
s_{23}=s_{24}=s_{34}=0.
\end{align*}
Then, this last set of zeros turns \eqref{12}, \eqref{13} and \eqref{14} into the new equations, 
\begin{align*}
&s_{11}\xi_{1}\bar{\rho}+s_{15}\xi_{1}\beta+s_{17}N_{12}(\xi;\bar{q})+s_{18}N_{13}(\xi;\bar{q})=\alpha s_{22}\xi_{1},\tag{1:2'}\label{12p}\\
&s_{11}\xi_{2}\bar{\rho}+s_{15}\xi_{2}\beta-s_{16}N_{12}(\xi;\bar{q})+s_{18}N_{23}(\xi;\bar{q})=\alpha s_{33}\xi_{2},\tag{1:3'}\label{13p}\\
&s_{11}\xi_{3}\bar{\rho}+s_{15}\xi_{3}\beta-s_{16}N_{13}(\xi;\bar{q})-s_{17}N_{23}(\xi;\bar{q})=\alpha s_{44}\xi_{3}.\tag{1:4'}\label{14p}
\end{align*}
Now, evaluate \eqref{12p} in $\xi=\hat{e}_{3}$, then in $\xi=\hat{e}_{2}$, and \eqref{13p} in $\xi=\hat{e}_{1}$, use \eqref{eq:simplenotation} and the fact that $\bar{q}_{i}\neq 0$ for all $i=1,2,3$ to obtain that $s_{16}=s_{17}=_{18}=0$. Then, by \eqref{26p}, \eqref{37p} and \eqref{48p}, $s_{25}=s_{35}=s_{45}=0$.
Finally, equations \eqref{38p}, \eqref{46p} and \eqref{47p} become, 
\begin{align}
s_{88}N_{23}(\xi;\bar{q})=0,\quad s_{66}N_{13}(\xi;\bar{q})=0,\quad s_{77}N_{23}(\xi;\bar{q})=0\quad\mbox{for all}\quad\xi\in\mathbb{S}^{2}.\label{eq:nonsymmetric}
\end{align}
Therefore, by \eqref{eq:simplenotation} and the assumption $q_{i}\neq 0$ for all $i=1,2,3$, we deduce that $s_{66}=s_{77}=s_{88}=0$. This is a contradiction with \eqref{eq:positivity}. This concludes the argument.
\end{proof}
Consider the following definition.
\begin{defi}
\label{symbolicsymm}
A symbolic symmetrizer associated with $A(\xi;V(x,t))$, is a smooth mapping $S:\mathbb{R}^{3}\times\mathbb{R}_{+}\times\left(\mathbb{R}^{3}\setminus\{0\}\right)\rightarrow M_{8}(\mathbb{C})$, homogeneous of degree 0 in its last variable $\xi$, bounded as well as all its derivatives with respect to $(x,t;\xi)$ on $|\xi|=1$, such that, for all $(x,t;\xi)$, 
\begin{align*}
	S(x,t;\xi)^{\ast}=S(x,t;\xi)\geq\delta(x,t;\xi),
\end{align*}
for some positive $\delta(x,t;\xi)$ with the property that, for any compact set $\mathcal{K}\in\mathbb{R}^{3}$,
\begin{align}
	\label{eq:locallower}
	\inf_{\substack{(x,t)\in\mathcal{K}\times[0,T]\\ |\xi|=1}}\delta(x,t;\xi)=:d(\mathcal{K})>0,
\end{align}
and 
\begin{align*}
	S(x,t;\xi)A(\xi;V(x,t))=A(\xi;V(x,t))^{\ast}S(x,t;\xi).
\end{align*}
\end{defi}
As consequence of the following lemmas, one can show the existence of a symbolic symmetrizer for system \eqref{eq:stronghyp}.
\begin{lema}
\label{constmultip}
Under the thermodynamical assumptions (\textbf{C1})-(\textbf{C2}), the algebraic multiplicity of the characteristic roots of system \eqref{eq:1}-\eqref{eq:objder} is independent of $U\in\mathcal{O}$ and $\xi\in\mathbb{S}^{2}$.
\end{lema}
\begin{proof}
For any  $U\in\mathcal{O}$ and $\xi\in\mathbb{S}^{2}$, the characterstic equation of system  \eqref{eq:1}-\eqref{eq:objder} is 
\begin{align*}
	\det\left(A(\xi;U)-\eta\mathbb{I}_{3}\right)=\rho^{3}\tau^{2}\left(\xi\cdot v-\eta\right)^{4}P_{0}(\xi;U,\eta)=0,
\end{align*}
where 
\begin{align*}
	P_{0}(\xi;U,\eta)=\rho e_{\theta}\tau\left(\xi\cdot v-\eta\right)^{4}-\left(\tau\rho e_{\theta}p_{\rho}+\frac{\tau\theta p_{\theta}^{2}}{\rho}\kappa\right)\left(\xi\cdot v-\eta\right)^{2}+\kappa p_{\rho},
\end{align*}
see \cite[Equations 5.5, 5.6 and Theorem 6.1]{angeleshyp}. Therefore, the algebraic multiplicity of the root $\eta_{0}(\xi;U)=v\cdot\xi$ is at least  equal four. From the equation $P_{0}(\xi;U,\eta)=0$ we obtain four real roots given by 
\begin{equation}
	\label{eq:realroots}
	\begin{aligned}
		&\eta_{1}(\xi;U)=v\cdot\xi+\sqrt{z_{+}^{2}},\quad\eta_{2}(\xi;U)=v\cdot\xi+\sqrt{z_{-}^{2}},\\
		&\eta_{3}(\xi;U)=v\cdot\xi-\sqrt{z_{+}^{2}},\quad\mbox{and}\quad\eta_{4}(\xi;U)=v\cdot\xi-\sqrt{z_{-}^{2}},
	\end{aligned}
\end{equation}
where $z_{+}^{2}$ and $z_{-}^{2}$ are defined as
\begin{equation}
	\label{eq:zetas}
	\begin{aligned}
		z_{+}^{2}&=\frac{1}{2}\left(p_{\rho}+\frac{\theta p_{\theta}^{2}}{\rho^{2}e_{\theta}}+\frac{\kappa}{\rho e_{\theta}\tau}\right)+\frac{1}{2}\sqrt{\left(p_{\rho}+\frac{\theta p_{\theta}^{2}}{\rho^{2}e_{\theta}}+\frac{\kappa}{\rho e_{\theta}\tau}\right)^{2}-\frac{4p_{\rho}\kappa}{\rho e_{\theta}\tau}},\\
		z_{-}^{2}&=\frac{1}{2}\left(p_{\rho}+\frac{\theta p_{\theta}^{2}}{\rho^{2}e_{\theta}}+\frac{\kappa}{\rho e_{\theta}\tau}\right)-\frac{1}{2}\sqrt{\left(p_{\rho}+\frac{\theta p_{\theta}^{2}}{\rho^{2}e_{\theta}}+\frac{\kappa}{\rho e_{\theta}\tau}\right)^{2}-\frac{4p_{\rho}\kappa}{\rho e_{\theta}\tau}},
	\end{aligned}
\end{equation}
see \cite[section 3]{angelesnon}. By the thermodynamical assumptions, $z_{+}^{2}>z_{-}^{2}>0$ and thus,
\begin{align}
	\eta_{3}(\xi;U)<\eta_{4}(\xi;U)<\eta_{0}(\xi;U)<\eta_{2}(\xi;U)<\eta_{1}(\xi;U),\quad\mbox{for every}\quad\xi\in\mathbb{S}^{2}\quad\mbox{and}\quad U\in\mathcal{O}.\label{eq:eigeninequality}
\end{align} 
In consequence, the algebraic multiplicity of each root in the set  $\left\lbrace\eta_{j}(\xi;U)\right\rbrace_{j=1}^{4}$ is equal to one, which in turn implies that $\eta_{0}(\xi;U)$ is a root of algebraic multiplicity equal to four, for any choice of $\xi\in\mathbb{S}^{2}$ and $U\in\mathcal{O}$.
\end{proof}
\begin{lema}
\label{lowests}
Let $V=V(x,t)\in\mathcal{O}$ be a continuous state for which there is an open bounded convex set $\mathcal{C}$ with the following properties, 
\begin{align}
	\label{eq:conseparation}
	\overline{\mathcal{C}}\subset\mathcal{O}\quad\mbox{and}\quad V(x,t)\in\mathcal{C},\quad\mbox{for all}\quad(x,t)\in\mathbb{R}^{3}\times[0,T]\quad\mbox{for some}\quad T>0.
\end{align}
If $V=(\rho,v,\theta,q)^{\top}$, define $z_{+}^{2}$ and $z_{-}^{2}$ as in \eqref{eq:zetas}. Then, there are positive constants $\delta_{1},\delta_{2},\delta_{3}$, $\delta_{4}$, depending only on $\mathcal{C}$ and the thermal relaxation $\tau$, for which the following inequalities hold true,
\begin{align}
\delta_{1}\leq& z_{+}^{2}\left(\rho(x,t),\theta(x,t)\right)\leq\delta_{2},\label{eq:zeta1}\\
\delta_{3}\leq&z_{-}^{2}\left(\rho(x,t),\theta(x,t)\right),\label{eq:zeta2}\\
\delta_{4}\leq&\sqrt{z_{+}^{2}\left(\rho(x,t),\theta(x,t)\right)}-\sqrt{z_{-}^{2}\left(\rho(x,t),\theta(x,t)\right)},\label{eq:zeta3}
\end{align}
for all $(x,t)\in\mathbb{R}^{3}\times[0,T]$.
\end{lema}
\begin{proof}
Assume \eqref{eq:conseparation} holds true for the state $V(x,t)=\left(\rho,v,\theta,q\right)(x,t)$. Since $\overline{\mathcal{C}}$ is a compact set contained in $\mathcal{O}$, there are positive constants $\rho_{0},\rho_{1},\theta_{0},\theta_{1}$ for which 
\begin{align*}
	\rho_{0}\leq \rho(x,t)\leq \rho_{1}\quad\mbox{and}\quad\theta_{0}\leq\theta(x,t)\leq\theta_{1}\quad\mbox{for all}\quad(x,t)\in\mathbb{R}^{3}\times[0,T].
\end{align*}
Then, by continuity of the state $V$ and the thermodynamical assumptions (\textbf{C1})-(\textbf{C2}), there are positive constants $M_{1}$ and $M_{2}$  such that,
\begin{align}
	\label{eq:continuity1} 
	M_{1}\leq p_{\rho}(\rho,\theta),p_{\theta}(\rho,\theta),e_{\theta}(\rho,\theta),\kappa(\rho,\theta)\leq M_{2}\quad\mbox{on the set}\quad\mathbb{R}^{3}\times[0,T].
\end{align}
Define the functions $\Theta,\Phi:\mathbb{R}^{3}\times[0,T]\rightarrow\mathbb{R}$ as 
\begin{align*}
	\Theta&:=p_{\rho}(\rho,\theta)+\frac{\theta p_{\theta}^{2}(\rho,\theta)}{\rho^{2}e_{\theta}(\rho,\theta)}+\frac{\kappa(\rho,\theta)}{\rho e_{\theta}(\rho,\theta)\tau},\\
	\Phi&:=\frac{4p_{\rho}(\rho,\theta)\kappa(\rho,\theta)}{\rho e_{\theta}(\rho,\theta)\tau}
\end{align*}
and notice that, by \eqref{eq:continuity1}, the inequalities
\begin{equation}
\label{eq:continuity2}
\begin{aligned}
		M_{1}+\frac{\theta_{0}M_{1}^{2}}{\rho_{1}^{2}M_{2}}+\frac{M_{1}}{\rho_{1}M_{2}\tau}\leq
	&\Theta(x,t)\leq M_{2}+\frac{\theta_{1}M_{2}^{2}}{\rho_{0}^{2}M_{1}}+\frac{M_{2}}{\rho_{0}M_{1}\tau}\\
	\frac{4M_{1}^{2}}{\rho_{1}M_{2}\tau}\leq&\Phi(x,t)\leq\frac{4M_{2}^{2}}{\rho_{0}M_{1}\tau},
\end{aligned}
\end{equation}
hold for all $(x,t)\in\mathbb{R}^{3}\times[0,T]$.
Since $\tfrac{1}{2}\Theta\leq z_{+}^{2}\leq\Theta$, \eqref{eq:zeta1} follows by taking
\begin{align*}
\delta_{1}:=\frac{1}{2}\left(	M_{1}+\frac{\theta_{0}M_{1}^{2}}{\rho_{1}^{2}M_{2}}+\frac{M_{1}}{\rho_{1}M_{2}\tau}\right)\quad\mbox{and}\quad \delta_{2}:=M_{2}+\frac{\theta_{1}M_{2}^{2}}{\rho_{0}^{2}M_{1}}+\frac{M_{2}}{\rho_{0}M_{1}\tau}.
\end{align*}
Observe that, 
 $z_{-}^{2}=\frac{1}{2}\left\lbrace\Theta-\sqrt{\Theta^{2}-\Phi}\right\rbrace$ where, 
\begin{align*}
	\Theta-\sqrt{\Theta^{2}-\Phi}=\int_{0}^{1}\frac{d}{dr}\sqrt{\Theta^{2}-\Phi+r\Phi}~dr=\frac{\Phi}{2}\int_{0}^{1}\frac{dr}{\sqrt{\Theta^{2}-\Phi+r\Phi}},
\end{align*}
and thus, 
\begin{align*}
	z_{-}^{2}\geq\frac{\Phi}{4}\int_{0}^{1}\frac{dr}{\sqrt{\Theta^{2}+\Phi}}=\frac{\Phi}{4\sqrt{\Theta^{2}+\Phi}}.
\end{align*}
Then, the inequality $\sqrt{\Theta^{2}+\Phi}\leq\sqrt{2}\left(\Theta+\sqrt{\Phi}\right)$, together with \eqref{eq:continuity2} imply that,
\begin{align*}
	z_{-}^{2}\left(\rho(x,t),\theta(x,t)\right)\geq\frac{M_{1}^{2}}{\sqrt{2}\rho_{1}M_{2}\tau}\left(M_{2}+\frac{\theta_{1}M_{2}^{2}}{\rho_{0}^{2}M_{1}}+\frac{M_{2}}{\rho_{0}M_{1}\tau}+\frac{2M_{2}}{\sqrt{\rho_{0}M_{1}\tau}}\right)^{-1}=:\delta_{3},
\end{align*}
for all $(x,t)\in\mathbb{R}^{3}\times[0,T]$.\\
In order to prove \eqref{eq:zeta3}, first notice that $z_{-}^{2}+r\left(z_{+}^{2}-z_{-}^{2}\right)\leq z_{+}^{2}$ holds for every $r\in[0,1]$. Hence, 
\begin{equation}
	\label{eq:lowest3}
	\begin{aligned}
		\sqrt{z_{+}^{2}}-\sqrt{z_{-}^{2}}=\int_{0}^{1}\frac{d}{dr}\sqrt{z_{-}^{2}+r(z_{+}^{2}-z_{-}^{2})}~dr&=\frac{z_{+}^{2}-z_{-}^{2}}{2}\int_{0}^{1}\frac{dr}{\sqrt{z_{-}^{2}+r(z_{+}^{2}-z_{-}^{2})}}\\
		&\geq\frac{z_{+}^{2}-z_{-}^{2}}{2\sqrt{z_{+}^{2}}}.
	\end{aligned}
\end{equation}
By \eqref{eq:zetas} and the standard inequality, $\left(p_{\rho}-\frac{\kappa}{\rho e_{\theta}\tau}\right)^{2}\geq 0$, it follows that
\begin{equation}
	\label{eq:lowest4}
	\begin{aligned}
		z_{+}^{2}-z_{-}^{2}&=\sqrt{p_{\rho}^{2}+\left(\frac{\kappa}{\rho e_{\theta}\tau}\right)^{2}-2\frac{p_{\rho}\kappa}{\rho e_{\theta}\tau}+\left(\frac{\theta p_{\theta}^{2}}{\rho^{2}e_{\theta}}\right)^{2}+\frac{2\theta p_{\theta}^{2}}{\rho^{2}e_{\theta}}\left(p_{\rho}+\frac{\kappa}{\rho e_{\theta}\tau}\right)}\\
		&\geq \sqrt{\left(\frac{\theta p_{\theta}^{2}}{\rho^{2}e_{\theta}}\right)^{2}+\frac{2\theta p_{\theta}^{2}}{\rho^{2}e_{\theta}}\left(p_{\rho}+\frac{\kappa}{\rho e_{\theta}\tau}\right)}\\
		&\geq\frac{\theta p_{\theta}^{2}}{\rho^{2}e_{\theta}}.
	\end{aligned}
\end{equation}
Therefore, as consequence of \eqref{eq:zeta1}, \eqref{eq:continuity1}, \eqref{eq:lowest3} and \eqref{eq:lowest4}, it is enough to take
\begin{align*}
	\delta_{4}:=\frac{\theta_{0}M_{1}^{2}}{2\rho_{1}^{2}M_{2}}\frac{1}{\sqrt{\delta_{2}}}.
\end{align*}
This completes the proof.
\end{proof}
\begin{lema}
\label{eigenseparation}
Let $V=V(x,t)\in\mathcal{O}$ be a continuous state satisfying \eqref{eq:conseparation}. If $\left\lbrace\eta_{j}(\xi;V(x,t))\right\rbrace_{j=0}^{4}$ is the set of eigenvalues of $A(\xi;V(x,t))$ then, there is a positive constant $\delta$, depending only on $\mathcal{C}$ and $\tau$, such that
\begin{align}
	\label{eq:spectralgap}
\inf_{\substack{(x,t)\in\mathbb{R}^{3}\times[0,T]\\ \xi\in\mathbb{R}^{3}\setminus\{0\}, j\neq k}}|\eta_{j}(\xi;V(x,t))-\eta_{k}(\xi;V(x,t))|\geq\delta.
\end{align}
\end{lema}
\begin{proof}
From \eqref{eq:realroots}, the following identities are true,
\begin{align*}
&|\eta_{0}-\eta_{1}|=|\eta_{0}-\eta_{3}|=\frac{1}{2}|\eta_{1}-\eta_{3}|=\sqrt{z_{+}^{2}},\quad|\eta_{0}-\eta_{2}|=|\eta_{0}-\eta_{4}|=\frac{1}{2}|\eta_{2}-\eta_{4}|=\sqrt{z_{-}^{2}},\\
&|\eta_{1}-\eta_{2}|=|\eta_{3}-\eta_{4}|=\sqrt{z_{+}^{2}}-\sqrt{z_{-}^{2}},\quad|\eta_{1}-\eta_{4}|=|\eta_{2}-\eta_{3}|=\sqrt{z_{+}^{2}}+\sqrt{z_{-}^{2}}.
\end{align*}
Then, the result is a direct consequence of \eqref{eq:zeta1}, \eqref{eq:zeta2} and \eqref{eq:zeta3} by taking $\delta:=\min\left\lbrace \delta_{1},\delta_{3},\delta_{4}\right\rbrace$.
\end{proof}
As consequence of the constant hyperbolicity of each linear system (Lemma \ref{constmultip}), the spectral separation (Lemma \ref{eigenseparation}) and the diagonalizability of the symbol $A(\xi;V(x,t))$ (see \cite{angeleshyp}), we can construct locally the symbolic symmetrizer as it is done in \cite{pseudofri} and \cite{kajitani}. 
Set $P(x,t;\lambda,\xi)$ as the determinant of $\lambda\mathbb{I}-A(\xi;V(x,t))$, $P_{\lambda}(x,t;\lambda,\xi)=\frac{\partial}{\partial\lambda}P(x,t;\lambda,\xi)$ and let $\Gamma$ be a Jordan curve which contains the characteristic roots of $A(\xi;V(x,t))$ but does not contain the poles of $PP_{\lambda}^{-1}(x,t;\lambda,\xi)$. Then, a symbolic symmetrizer is given by the formula
\begin{align*}
	S(x,t;\xi)=\frac{1}{2\pi i}\oint_{\Gamma}\left(\lambda\mathbb{I}-A(\xi;V(x,t))^{\ast}\right)^{-1}\left(\lambda\mathbb{I}-A(\xi;V(x,t))\right)^{-1}P_{\lambda}(x,t;\lambda,\xi)^{-1}P(x,t;\lambda,\xi)d\lambda.
\end{align*}
In \cite{pseudofri}, Friedrichs shows that $S(x,t,\xi)$ satisfies the propierties of Definition \ref{symbolicsymm}. Moreover, by the uniform lower bound in the spectral gap (i.e. \eqref{eq:spectralgap}), instead of \eqref{eq:locallower}, it can be shown that
\begin{align*}
	\inf_{\substack{(x,t)\in\mathbb{R}^{3}\times[0,T]\\ |\xi|=1}}\delta(x,t;\xi)=:d_{0}>0,
\end{align*}
see \cite{kajitani} and the references therein. If $V\in\mathcal{O}$ is constant or independent of $x$, for sufficiently large values of $|x|$, the existence of a symbolic symmetrizer can also be found in \cite[Theorem 2.3]{benzoserre}. Furthermore, the smoothness of $V$ is not necessary for the existence of a symbolic symmetrizer, see \cite{metivierl2} and the references therein.\\
Another consequence of Lemmas \ref{constmultip} and \ref{eigenseparation} is the following result, stated without proof.

\begin{theo}[Kajitani \cite{kajitani}]
	\label{kajitanil2}
Consider the linear system of equations \eqref{eq:stronghyp}. Assume that the characteristic roots of $\sum_{j}\widetilde{A}^{j}(x,t)\xi_{j}$ are real, of constant algebraic multiplicity and satisfy \eqref{eq:spectralgap}. Then, the following statements are true:
\begin{itemize}
\item [(a)] The Cauchy problem for \eqref{eq:stronghyp} is $L^{2}$-well posed.
\item [(b)] The symbol matrix $\widehat{A}(\xi;x,t):=\sum_{j}\widetilde{A}^{j}(x,t)\xi_{j}$ is diagonalizable for any $(x,t;\xi)\in\mathbb{R}^{d}\times[0,T]\times\mathbb{S}^{d-1}$.
\end{itemize}
\qed
\end{theo}
\begin{theo}
Let $V\in\mathcal{O}$ be a smooth state for which, \eqref{eq:conseparation} is satisfied. Then, the Cauchy problem 
\begin{equation}
\label{eq:linearcauchy}
\begin{aligned}
U_{t}+A^{j}(V)\partial_{j}U+DQ(V)U&=0,\\
U\big\vert_{t=0}&=U_{0},
\end{aligned}
\end{equation}
is $L^{2}$-well posed. In particular, the result holds true for any constant state $V_{c}\in\mathcal{O}$.
\end{theo}
\begin{proof}
First notice that, by (\textbf{C1}), the mappings $\mathcal{O}\ni W\mapsto A^{j}(W)$ ($j=1,2,3$) and $\ni W\mapsto DQ(W)$ are smooth and thus, by the smoothness of $V\in\mathcal{O}$, the matrices
\begin{align*}
\widetilde{A}^{j}(x,t):=A^{j}(V(x,t))\quad\mbox{and}\quad\widetilde{D}(x,t):=DQ(V(x,t)),
\end{align*}
are infinitely differentiable on the set $(x,t)\in\mathbb{R}^{3}\times(0,T)$. By \cite[Theorem 6.1]{angeleshyp} and Lemma \ref{constmultip} the characteristic roots of $\widetilde{A}(\xi;x,t)$ are real, of constant algebraic multiplicity and given by \eqref{eq:realroots}. Furthermore, since $V\in\mathcal{O}$ satisfies \eqref{eq:conseparation}, inequality \eqref{eq:spectralgap} holds true. By \cite[Theorem 6.2]{angeleshyp}, the symbol matrix $\widetilde{A}(\xi;x,t)$ is diagonalizable and thus, Theorem \ref{kajitanil2} implies the $L^{2}$-well posedness of \eqref{eq:linearcauchy}.\\
Finally, assume that $V=V_{c}\in\mathcal{O}$ is a constant state. Then, $V_{c}$ is trivially smooth. Since $\mathcal{O}$ is an open set, there is an open ball centered at $V_{c}$, let's say $B_{R}(V_{c})$, and such that, $\overline{B_{R}(V_{c})}\subset\mathcal{O}$. Then, take $\mathcal{C}:=B_{R}(V_{c})$ and observe that condition \eqref{eq:conseparation} is satisfied for any $T>0$. This concludes the argument.
\end{proof}
\begin{rem}
Consider the Cauchy problem for \eqref{eq:quasi2} with initial data $U_{0}$. If $U_{0}\in H^{s}$, for sufficiently large $s$, then, the local solution to the problem satisfies condition \eqref{eq:conseparation} (see \cite[Section 7]{angelescauchy}, for example). Therefore, the importance of assuming condition \eqref{eq:conseparation} is that, the  solution to the quasilinear problem is typically understood as the unique fixed point of the mapping that sends $V\in\mathcal{O}$ to the solution of \eqref{eq:linearcauchy}.
\end{rem}
\section{Linearization around equlibrium states}
For each $U\in\mathcal{O}$, we can decompose the symbol $A(\xi;U)$ as 
\begin{align*}
A(\xi;U)=A_{0}(\xi;U)+N(\xi;U),
\end{align*}
where 
	\small
\begin{align}
	\label{eq:A0xi}
	A_{0}(\xi;U)=\left(\begin{array}{cccccccc}
		\xi\cdot v&\xi_{1}\rho&\xi_{2}\rho&\xi_{3}\rho&0&0&0&0\\
		\xi_{1}\frac{p_{\rho}}{\rho}&\xi\cdot v&0&0&\xi_{1}\frac{p_{\theta}}{\rho}&0&0&0\\
		\xi_{2}\frac{p_{\rho}}{\rho}&0&\xi\cdot v&0&\xi_{2}\frac{p_{\theta}}{\rho}&0&0&0\\
		\xi_{3}\frac{p_{\rho}}{\rho}&0&0&\xi\cdot v&\xi_{3}\frac{p_{\theta}}{\rho}&0&0&0\\
		0&\xi_{1}\frac{\theta p_{\theta}}{\rho e_{\theta}}&\xi_{2}\frac{\theta p_{\theta}}{\rho e_{\theta}}&\xi_{3}\frac{\theta p_{\theta}}{\rho e_{\theta}}&v\cdot \xi&\frac{\xi_{1}}{\rho e_{\theta}}&\frac{\xi_{2}}{\rho e_{\theta}}&\frac{\xi_{3}}{\rho e_{\theta}}\\
		0&&&&\frac{\xi_{1}\kappa}{\tau}&v\cdot\xi&0&0\\
		0&&\mathbb{O}_{3\times 3}&&\frac{\xi_{2}\kappa}{\tau}&0&v\cdot\xi&0\\
		0&&&&\frac{\xi_{3}\kappa}{\tau}&0&0&v\cdot\xi\\
	\end{array}\right)
\end{align}
\normalsize
and
	\small
\begin{align}
	\label{eq:Nxi}
	N(\xi;U)=\left(\begin{array}{cccccccc}
		0&0&0&0&0&0&0&0\\
		0&0&0&0&0&0&0&0\\
		0&0&0&0&0&0&0&0\\
		0&0&0&0&0&0&0&0\\
		0&0&0&0&0&0&0&0\\
		0&0&0&0&0&0&0&0\\
		0&&\mathcal{Q}_{1,-1}(\xi;q)&0&0&0&0&0\\
		0&&&0&0&0&0&0\\
	\end{array}\right).
\end{align}
\normalsize
The matrix $A_{0}(\xi;U)$ is the symbol associated with the quasilinear form of the Cattaneo-Christov-Jordan system, that is, the coupling between equations \eqref{eq:1}-\eqref{eq:3} and \eqref{eq:materialflux} (\cite[section 6]{angeleshyp}, \cite{jordy}). Therefore, the system can be written in the form 
\begin{align}
\label{eq:ccjsys}
U_{t}+A^{j}_{0}(U)\partial_{j}U+Q(U)=0,
\end{align}
with $A^{j}_{0}(U):=A_{0}(e_{j};U)$ and where $\{e_{j}\}_{j=1}^{3}$ denotes the canonical vector basis in $\mathbb{R}^{3}$. It is easy to verify that, the diagonal matrix
\small
\begin{align}
	S_{0}(U)=\left(\begin{array}{cccc}
		\frac{p_{\rho}}{\rho^{2}}&&&\\
	&\mathbb{I}_{3\times 3}&&\\
		&&\frac{e_{\theta}}{\theta}&\\
		&&&\frac{\tau}{\kappa\rho\theta}\mathbb{I}_{3\times 3}\label{eq:symmetrizerA0}
	\end{array}\right),\quad U\in\mathcal{O},
\end{align}
\normalsize
defines a Friedrichs symmetrizer for the quasilinear system \eqref{eq:ccjsys}. Hence, it is a hyperbolic system of equations and its Cauchy problem is $L^{2}$-well posed \cite{katosym}.\\
The Cattaneo-Christov-Jordan system is related with the linearization of \eqref{eq:quasi2} around \emph{equilibrium states}. Consider the following definition (see \cite[section 2]{mascialini}).
\begin{defi}
	The equilibrium manifold of a quasilinear system of the form \eqref{eq:quasi2} is defined as 
	\begin{align*}
		\mathcal{V}=\left\lbrace V\in\mathcal{O}~|~Q(V)=0\right\rbrace.
	\end{align*}
	A function $W=W(x,t)$ is an equilibrium solution of \eqref{eq:quasi2} or Maxwellian, if it lies in the equilibrium manifold, that is $Q(W(x,t))=0$ for any $(x,t)$ and satisfies 
	\begin{align}
		U_{t}+A^{i}(W)\partial_{i}U=0.\label{eq:eqsol}
	\end{align}
By a constant equilibrium state we mean a state $V_{e}\in\mathcal{O}$ that is independent of $x\in\mathbb{R}^{3}$ and $t>0$, that lies in $\mathcal{V}$.
\end{defi}
From \eqref{eq:QU2}, it follows that
\begin{align}
	\label{eq:equilibriumchar}
	\mathcal{V}=\left\lbrace(\rho,v,\theta,q)^{\top}\in\mathcal{O}~\vert~q=0\right\rbrace.
\end{align}
Let $V_{c}\in\mathcal{V}$ be a constant equilibrium state and assume that $V_{c}+W$ is a solution of \eqref{eq:quasi2}. Then, we can recast \eqref{eq:quasi2} as 
\begin{align}
W_{t}+A^{j}(V_{e})\partial_{j}W+Q(V_{e})+DQ(V_{e})W+\mathcal{N}\left(W,D_{x}W,D^{2}_{x}W_{t}\right)=0,\label{eq:taylor}
\end{align}
where $\mathcal{N}\left(W,D_{x}W,D^{2}_{x}W,_{t}\right)$ comprises all the nonlinear terms. By discarding the non-linear terms and using that $Q(V_{e})=0$, we obtain the \emph{linearization} of \eqref{eq:quasi2} around constant equilibrium states, namely,
\begin{align}
W_{t}+A^{j}(V_{e})\partial_{j}W+DQ(V_{e})W=0.\label{eq:lineartaylor1}
\end{align}
However, since $N(\xi;V_{e})=0$,
\begin{align}
	\label{eq:symbolseq}
	A(\xi;V_{c})=A_{0}(\xi;V_{c}),\quad\mbox{for all}\quad\xi\in\mathbb{R}^{3}\setminus\{0\}.
\end{align}
Therefore, system \eqref{eq:lineartaylor1} is the same as 
\begin{align}
	W_{t}+A^{j}_{0}(V_{e})\partial_{j}W+DQ(V_{e})W=0.\label{eq:lineartaylor2}
\end{align}
We have the following result.
\begin{res}
The linearizations of the systems \eqref{eq:quasi2} and \eqref{eq:ccjsys} around constant equilibrium states coincide.
\end{res}
As consequence of \eqref{eq:equilibriumchar}, for any constant equilibrium state $V_{e}\in\mathcal{V}$, there are constant values of $\rho_{e}>0$, $\theta_{e}>0$ and $v_{e}\in\mathbb{R}^{3}$ such that $V_{e}=\left(\rho_{e},v_{e},\theta_{e},0,0,0\right)^{\top}$. Define the positive constants
\begin{align*}
	L_{0}:=\min\left\lbrace 1,\frac{p_{\rho}(\rho_{e},\theta_{e})}{\rho^{2}_{e}},\frac{\tau}{\kappa(\rho_{e},\theta_{e})\rho_{e}\theta_{e}}\right\rbrace\quad\mbox{and}\quad L_{1}:=\max\left\lbrace 1,\frac{p_{\rho}(\rho_{e},\theta_{e})}{\rho^{2}_{e}},\frac{\tau}{\kappa(\rho_{e},\theta_{e})\rho_{e}\theta_{e}}\right\rbrace.
\end{align*}
Then, since \eqref{eq:symmetrizerA0} symmetrizes the quasilinear Cattaneo-Christov-Jordan system, the constant matrix $S_{0}(V_{e})$ satisfies that
\begin{equation}
	\label{eq:symwell}
	\begin{aligned}
		&L_{0}|Z|^{2}\leq\left(S_{0}(V_{e})Z,Z\right)_{\mathbb{R}^{8}}\leq L_{1}|Z|^{2}\quad\mbox{for all}\quad Z\in\mathbb{R}^{8},\\
		&S_{0}(V_{e})A_{0}(\xi;V_{e})=A_{0}(\xi;V_{e})^{\top}S_{0}(V_{e})\quad\mbox{for all}\quad\xi\in\mathbb{R}^{3}\setminus\{0\}.
	\end{aligned}
\end{equation}
Hence, by Remark \ref{remlinearf},  $S_{0}(V_{e})$ is a Friedrichs symmetrizer of \eqref{eq:lineartaylor2}. Therefore, the Cauchy problem for \eqref{eq:lineartaylor2} with initial data $W_{0}\in L^{2}(\mathbb{R}^{3})$ given at $t=0$, is $L^{2}$-well posed (see \cite[Chapter 2]{otto}, \cite[Section 3.1]{serre}).
\begin{rem}
\label{remkawasym}
Observe that $S(V_{e})DQ(V_{e})$ is symmetric and therefore, the system 
\begin{align}
S(V_{e})W_{t}+S(V_{e})A^{j}(V_{e})\partial_{j}W+S(V_{e})DQ(V_{e})W=0,\label{eq:symmlinear}
\end{align}
is symmetric hyperbolic. Although the symmetrizability of $DQ(V_{e})$ is not needed to conclude the $L^{2}$ well-posedness of \eqref{eq:lineartaylor2}, it is required for the application of the Kawashima-Shizuta equivalence Theorem \cite[Theorem 1.1]{kawazuta} in the next section.

\end{rem}
\section{Existence of persistent waves}
We begin this section by recalling the notion of strict dissipativity and the Kawashima-Shizuta genuinely coupling condition for \eqref{eq:lineartaylor2} (cf. \cite{kawazuta}).
\begin{defi}
\label{genuinelycoupling}
\begin{itemize}
	\item [(1)] Let $\lambda=\lambda(\xi)$ denote the eigenvalues of the characteristic equation
	\begin{align*}
		\det\left\lvert\lambda\mathbb{I}_{3}+iA(\xi;V_{e})+DQ(V_{e})\right\rvert=0.
	\end{align*}
	System \eqref{eq:lineartaylor2} is said to be strictly dissipative if $\operatorname{Re}\lambda(\xi)<0$ for all $\xi\in\mathbb{R}\setminus\{0\}$.
	\item [(2)] System \eqref{eq:lineartaylor2} satisfies the Kawashima-Shizuta genuinely copling condition at $V_{e}\in\mathcal{V}$ if for any $Z\in\mathbb{R}^{8}\setminus\{0\}$ with $DQ(V_{e})Z=0$ and any $\xi\in\mathbb{R}^{3}\setminus\{0\}$, $A(\xi;V_{e})Z\neq \mu Z$ for every $\mu\in\mathbb{R}$.
\end{itemize}
\end{defi}
\begin{rem}
\label{kawaequiv}
As consequence of the existence of the symmetrizer $S(V_{e})$, the strict dissipativity of system \eqref{eq:symmlinear} and that of \eqref{eq:lineartaylor2} are equivalent. For the same reason, condition $(2)$ holds for \eqref{eq:lineartaylor2} if and only if, it holds for \eqref{eq:symmlinear}.
\end{rem}
\begin{theo}
\label{openkawa}
\begin{itemize}
	\item [(i)] The linearization of \eqref{eq:quasi2} around any constant equilibrium state $V_{e}\in\mathcal{V}$ doesn't satisfy the Kawashima-Shizuta genuinely coupling condition.
	\item [(ii)] Every $\bar{\xi}\in\mathbb{R}^{3}\setminus\{0\}$ has an open neighborhood $\Omega$ and smooth functions $Z:\Omega\rightarrow\mathbb{R}^{8}\setminus\{0\}$ and $\mu:\Omega\rightarrow\mathbb{R}\setminus\{0\}$ such that $Z(\xi)\in\ker DQ(V_{e})$ and $A(\xi;V_{e})Z(\xi)=\mu(\xi)Z(\xi)$ for all $\xi\in\Omega$. 
\end{itemize}
\end{theo}
\begin{proof}
First notice that
\begin{align*}
\ker DQ(V_{e})=\left\lbrace Z=(Z_{i})_{i=1}^{8}\in\mathbb{R}^{8}~|~Z_{6}=Z_{7}=Z_{8}=0\right\rbrace.
\end{align*}
Let $\xi=(\xi_{1},\xi_{2},\xi_{3})\in\mathbb{R}^{3}\setminus\{0\}$ be an arbitrary vector. For $Z\in\ker DQ(V_{e})$ and some $\mu\in\mathbb{R}$ consider the eigenvalue problem 
\begin{align}
A(\xi;V_{e})Z=\mu Z.\label{eq:eigenvaluep}
\end{align}
Since $V_{e}=\left(\rho_{e},v_{e},\theta_{e},0,0,0\right)^{\top}$ lies in the equilibrium manifold, \eqref{eq:symbolseq} is valid and the linearization of \eqref{eq:quasi2} around $V_{e}$ is given by \eqref{eq:lineartaylor2}. Hence, the components of equation \eqref{eq:eigenvaluep} are 
\begin{equation}
\label{eq:kawabonga}
\begin{aligned}
\xi\cdot v_{e} Z_{1}+\rho_{e}\left(\xi_{1},\xi_{2},\xi_{3}\right)\cdot(Z_{2},Z_{3},Z_{4})&=\mu Z_{1},\\
\xi_{1}\alpha Z_{1}+\xi\cdot v_{e}Z_{2}+\xi_{1}\eta Z_{5}&=\mu Z_{2},\\
\xi_{2}\alpha Z_{1}+\xi\cdot v_{e}Z_{3}+\xi_{2}\eta Z_{5}&=\mu Z_{3},\\
\xi_{3}\alpha Z_{1}+\xi\cdot v_{e}Z_{4}+\xi_{3}\eta Z_{5}&=\mu Z_{4},\\
\beta\left[\left(\xi_{1},\xi_{2},\xi_{3}\right)\cdot\left(Z_{2},Z_{3},Z_{4}\right)\right]+\xi\cdot v_{e} Z_{5}&=\mu Z_{5},\\
\delta\xi_{1}Z_{5}&=0,\\
\delta\xi_{2}Z_{5}&=0,\\
\delta\xi_{3}Z_{5}&=0,
\end{aligned}
\end{equation}
where we have used the same notation as in \eqref{eq:simplenotation} but evaluated at the constant value $(\rho_{e},\theta_{e})$ instead of $(\bar{\rho},\bar{\theta})$.
For any given $\xi\in\mathbb{R}^{3}\setminus\{0\}$ define the nonempty set
\begin{align*}
\mathcal{K}(\xi):=\left\lbrace Z=(Z_{i})_{i=1}^{8}\in\ker DQ(V_{e})\setminus\{0\}~|~Z_{1}=Z_{5}=0\quad\mbox{and}\quad(\xi_{1},\xi_{2},\xi_{3})\cdot(Z_{2},Z_{3},Z_{4})=0\right\rbrace.
\end{align*}
Then, by taking $\mu(\xi)=\xi\cdot v_{e}$ and $Z(\xi)\in\mathcal{K}(\xi)$, equations \eqref{eq:kawabonga} are satisfied. Since every element in $\mathcal{K}(\xi)$ is nonzero, the Kawashima-Shizuta condition is violated for any $\xi\in\mathbb{R}^{3}\setminus\{0\}$. This proves the first statement of the Theorem.
\vspace{0.3cm}\\
Now fix $\bar\xi\in\mathbb{R}^{3}\setminus\{0\}$. There are vectors $\bar{z},z^{\prime}\in\mathbb{R}^{3}\setminus\{0\}$ such that 
\begin{align}
	\bar{\xi}\cdot\bar{z}=\bar{\xi}\cdot z^{\prime}=\bar{z}\cdot z^{\prime}=0\quad\mbox{and}\quad|\bar{z}|=|\bar{\xi}|^{-1}.\label{eq:licondition}
\end{align}
Define the function $F:\mathbb{R}^{3}\times\mathbb{R}^{3}\rightarrow\mathbb{R}^{3}$ as 
\begin{align*}
	F(\xi,z)=\left(\begin{array}{c}
		\xi\cdot z\\
		|\bar{\xi}|^{2}|z|^{2}-1\\
		z^{\prime}\cdot z
	\end{array}\right).
\end{align*}
Then, $F\in\mathcal{C}^{\infty}\left(\mathbb{R}^{3}\times\mathbb{R}^{3};\mathbb{R}^{3}\right)$ and \eqref{eq:licondition} assures that $F(\bar{\xi},\bar{z})=0$. Moreover, 
\begin{align*}
	D_{z}F(\bar{\xi},\bar{z})=\left(\begin{array}{ccc}
		\bar{\xi_{1}}&\bar{\xi_{2}}&\bar{\xi_{3}}\\
		2|\bar{\xi}|^{2}\bar{z}_{1}&2|\bar{\xi}|^{2}\bar{z}_{2}&2|\bar{\xi}|^{2}\bar{z}_{3}\\
		z_{1}^{\prime}&z_{2}^{\prime}&z_{3}^{\prime}
	\end{array}\right)
\end{align*}
and thus, $\det D_{z}F(\bar{\xi},\bar{z})\neq 0$ since, by \eqref{eq:licondition}, the vectors $\bar{\xi}$, $\bar{z}$ and $z^{\prime}$ are linearly independent. By the implicit function theorem, there is an open neighborhood $\Omega$ of $\bar{\xi}$, and a mapping $h\in\mathcal{C}^{\infty}(\Omega;\mathbb{R}^{3})$ such that 
\begin{align*}
	F(\xi;h(\xi))=0\quad\mbox{and}\quad h(\xi)\neq 0\quad\mbox{for all}\quad\xi\in\Omega.
\end{align*}
Define the smooth mappings $Z:\Omega\rightarrow\mathbb{R}^{8}\setminus\{0\}$ and $\mu:\Omega\rightarrow\mathbb{R}\setminus\{0\}$ as
\begin{align}
\label{eq:smoothkawa}
Z(\xi):=\left(0,h(\xi),0,0,0,0\right)^{\top}\quad\mbox{and}\quad \mu:=\xi\cdot v_{e}
\end{align}
and observe that, for any $\xi\in\Omega$, $Z(\xi)\in\mathcal{K}(\xi)$ and the pair $(\mu(\xi),Z(\xi))$ satisfies \eqref{eq:kawabonga}, as claimed.
\end{proof}
\begin{coro}
Let $V_{e}\in\mathcal{V}$ be a constant equilibrium state. Then, system \eqref{eq:lineartaylor2} is not strictly dissipative.
\end{coro}
\begin{proof}
By Remark \ref{remkawasym}  and Theorem 1.1 in \cite{kawazuta} the strict dissipativity of \eqref{eq:symmlinear} is equivalent to its genuinely coupling condition. The result follows by Remark \ref{kawaequiv} and Theorem \ref{openkawa}.
\end{proof}
\begin{theo}[Existence of persistent waves]
\label{persistent}
The linearization of \eqref{eq:quasi2} around any constant equilibrium state $V_{e}\in\mathcal{V}$ has solutions $W$ with the following properties:
\begin{itemize}
	\item [(a)] $W$ is a smooth non-constant function satisfying the differential equation
	\begin{align*}
	W_{t}+A^{j}(V_{e})\partial_{j}W=0,
	\end{align*}
and the condition $ DQ(V_{e})W(x,t)=0$ for all $x\in\mathbb{R}^{3}$ and $t\geq 0$.
\item [(b)] It holds that 
\begin{align*}
\|W(\cdot,t)\|_{L^{2}(\mathbb{R}^{3})}=\|W(\cdot,0)\|_{L^{2}(\mathbb{R}^{3})}\quad\mbox{for all}\quad t\geq 0.
\end{align*}
\end{itemize}
\end{theo}
\begin{proof}
Let $\bar{\xi}\in\mathbb{R}^{3}\setminus\{0\}$ be fixed. Set $\Omega$ as the open neighborhood of $\bar{\xi}$ described in statement $(ii)$ of Theorem \ref{openkawa} and $(\mu,Z)$ the corresponding smooth mappings defined in \eqref{eq:smoothkawa}.
Let $\mathcal{B}$ be a closed ball containing $\bar{\xi}$ and $\Omega_{0}$ an open set in $\Omega$ such that $B\subset\Omega_{0}\subset\subset\Omega$,
and consider a cut off function $\phi\in\mathcal{C}_{0}^{\infty}(\Omega;\mathbb{R})$ such that
\begin{align*}
	\phi(\xi)=\left\lbrace\begin{array}{cc}
		1&\xi\in\mathcal{B},\\
		0&\xi\in\Omega\setminus\Omega_{0}.
	\end{array}\right.
\end{align*}
Then, the vector field $V_{0}:\mathbb{R}^{3}\rightarrow\mathbb{R}^{8}$ defined as
\begin{align*}
	V_{0}(\xi):=\left\lbrace\begin{array}{cc}
		\phi(\xi)Z(\xi)&\xi\in\Omega,\\
		0&\xi\in\mathbb{R}^{3}\setminus\Omega,
	\end{array}\right.
\end{align*}
is of class $\mathcal{C}_{0}^{\infty}(\mathbb{R}^{3};\mathbb{R}^{8})$. Consider the Cauchy problem
\begin{equation}
	\label{eq:sysper}
	\begin{aligned}
		W_{t}+A^{j}(V_{c})\partial_{j}W&=0,\\
		W(\cdot,0)&=\mathcal{F}^{-1}(V_{0}),
	\end{aligned}
\end{equation}
where $\mathcal{F}^{-1}(V_{0})$ denotes de inverse Fourier transform of $V_{0}$. By taking the Fourier transform of this problem and using \eqref{eq:symbolseq} we obtain the equivalent problem
\begin{align*}
	\widehat{W}_{t}+iA_{0}(\xi;V_{c})\widehat{W}&=0,\\
	\widehat{W}(\xi,0)&=V_{0}(\xi)\quad\mbox{for any}\quad\xi\in\mathbb{R}^{3}.
\end{align*}
The system in \eqref{eq:sysper} is of constant coefficients with a Friedrichs symmetrizer given by $S_{0}(V_{e})$. Therefore, its Cauchy problem is $L^{2}$-well posed and the unique solution satisfies that
\begin{align*}
	\widehat{W}(\xi,t)=e^{itA(\xi;V_{c})}V_{0}(\xi)=\left\lbrace\begin{array}{cc}
		e^{itA(\xi;V_{c})}\phi(\xi)Z(\xi)&\xi\in\Omega,\\
		0&\xi\in\mathbb{R}^{3}\setminus\Omega.
	\end{array}\right.
\end{align*}
The second statement in Theorem \ref{openkawa} assures that
\begin{align*}
	DQ(V_{c})\widehat{W}(\xi,t)&=\left\lbrace\begin{array}{cc}
		\phi(\xi)DQ(V_{c})e^{itA(\xi;V_{c})}Z(\xi)&\xi\in\Omega,\\
		0&\xi\in\mathbb{R}^{3}\setminus\Omega,
	\end{array}\right.\\
	&=\left\lbrace\begin{array}{cc}
		\phi(\xi)DQ(V_{c})e^{it\mu(\xi)}Z(\xi)&\xi\in\Omega,\\
		0&\xi\in\mathbb{R}^{3}\setminus\Omega,
	\end{array}\right.\\
	&=\left\lbrace\begin{array}{cc}
		\phi(\xi)e^{it\mu(\xi)}DQ(V_{c})Z(\xi)&\xi\in\Omega,\\
		0&\xi\in\mathbb{R}^{3}\setminus\Omega,
	\end{array}\right.
\end{align*}
where $\mu(\xi)=\xi\cdot v_{c}$.
By construction, $Z(\xi)\in\ker DQ(V_{c})$ for any $\xi\in\Omega$ and thus
\begin{align*}
	DQ(V_{c})\widehat{W}(\xi,t)=0\quad\mbox{for all}\quad\xi\in\mathbb{R}^{3},\quad t\geq 0.
\end{align*}
Therefore, by taking the inverse Fourier transform in the previous identity, we conclude that,
\begin{align}
	\label{eq:persistentwave}
	W(x,t)=\mathcal{F}^{-1}\left[e^{-i\xi\cdot(-tv_{c})}V_{0}(\xi)\right]=\mathcal{F}^{-1}(V_{0})(x+v_{c}t),
\end{align} 
complies with the properties of statement $(a)$. Moreover, by the Fourier-Plancherel formula, 
\begin{align*}
	\|W(\cdot,t)\|_{L^{2}(\mathbb{R}^{3})}^{2}&=(2\pi)^{-3}\|\widehat{W}(\cdot,t)\|_{L^{2}(\mathbb{R}^{3})}^{2}\\
	&=(2\pi)^{-3}\int_{\mathbb{R}^{3}}\left\lvert e^{it\xi\cdot v_{c}}V_{0}(\xi)\right\rvert^{2}d\xi\\
	&=(2\pi)^{-3}\|V_{0}\|_{L^{2}(\mathbb{R}^{3})}^{2}=\|\mathcal{F}^{-1}(V_{0})\|_{L^{2}(\mathbb{R}^{3})}^{2}\quad\mbox{for all}\quad t\geq 0.
\end{align*}
This concludes the proof.
\end{proof}
\section{Final comments and conclusions}
In this work we have shown that the only hyperbolic system of equations in compressible fluid dynamics involving an objective Cattaneo-type equation for the heat flux is not Friedrichs symmetrizable. We believe that this is an importat example for the literature in hyperbolic systems of equations in several dimensions. We have shown the existence of microlocal symmetrizers associated with the symbol $A(\xi;V(x,t))$ for smooth $V\in\mathcal{O}$ that comply with condition \eqref{eq:conseparation}. In this case, the $L^{2}$ energy estimates have to be obtained through the action of pseudodifferential operators. It remains to verify, if for less regular states $V\in\mathcal{O}$, the symbolic symmetrizer is Lipschitz continuous in $(x,t;\xi)\in\mathbb{R}^{3}\times\mathbb{R}_{+}\times\mathbb{S}^{2}$. Then, the energy estimates will be obtained through para-differential calculus (see, \cite{metivierl2} and \cite{metivierpara}). It is important to mention that, at this moment, the author of this work has not been capable of finding the explicit form of the symbolic symmetrizers.\\
Although, system \eqref{eq:quasi2} is not Friedrichs symmetrizable, its linearizations around states $V=(\rho,v,\theta,0,0,0)\in\mathcal{O}$ have Friedrichs symmetrizers. This led us to apply the Kawashima-Shizuta theory to the linearizations around constant equilibrium states in order to find out that the strict dissipativity of such systems is not satisfied. Now, for systems with this feature, it is always possible to find non-constant Maxwellians in the form travelling plane waves (see \cite[p. 735]{mascialini}). Indeed,  assume the existence of $\xi_{0}\in\mathbb{R}^{3}\setminus\{0\}$, $Z_{0}\in\mathbb{R}^{8}\setminus\{0\}$ and $\mu\in\mathbb{R}$ such that $DQ(V_{e})Z_{0}=0$ and $A(\xi_{0};V_{e})Z_{0}=\mu Z_{0}$. Then, for any smooth scalar function $H$, there is a non-constant Maxwellian in the form $W(x,t)=H(x\cdot\xi_{0}-\mu t)Z_{0}$. In contrast, we can construct non-constant Maxwellians in the form of waves by taking advantage of the violation of the Kawashima-Shizuta condition in open sets with respect to the Fourier frequency space.\\
It is well known that, the existence of a convex entropy plus the Kawashima-Shizuta condition guarantee stability of constant states for systems of hyperbolic-parabolic balance laws (see the introduction in \cite{mascialini}). There are many physical systems, still possessing dissipative entropies for which the Kawashima-Shizuta condition does not hold (see, \cite{zengas}, for example). This motivated Mascia and Natalini \cite{mascialini} in determining minimal conditions assuring stability of such equilibrium states. However, their results still require the conservative structure of the system and the existence of a dissipative entropy. These features doesn't seem to hold true for the quasilinear system \eqref{eq:quasi2}. This raises a new question: How far can we go (if possible) in the stability analysis without the Kawashima-Shizuta condition and the conservative structure?
\section*{Acknowledgements}
Thanks to professor Ram\'on Plaza, for his support in writing this work and to IIMAS UNAM, for their hospitality these past months. Thanks to Pedro Jordan for his trust in my work. This work was partially supported by CONAHCyT through a postdoctoral fellowship under grant CF-2023-G-122.

	\bibliographystyle{plain} 
	\bibliography{nonsymest}

\begin{thebibliography}{10}

\bibitem{angelesnon}
F.~Angeles.
\newblock Non-hyperbolicity of the inviscid cattaneo--christov system for
  compressible fluid flow in several space dimensions.
\newblock {\em The Quarterly Journal of Mechanics and Applied Mathematics},
  75(2):147--170, 2022.

\bibitem{angelescauchy}
F.~Angeles.
\newblock The {C}auchy problem for a quasilinear system of equations with
  coupling in the linearization.
\newblock {\em Commun. Pure Appl. Anal.}, 22(10):2960--2999, 2023.

\bibitem{angeleshyp}
F.~Angeles.
\newblock Hyperbolic systems of quasilinear equations in compressible fluid
  dynamics with an objective cattaneo-type extension for the heat flux.
\newblock {\em Mech. Res. Commun.}, 130:104103, 2023.

\bibitem{amp}
F.~Angeles, C.~M\'alaga, and R.~Plaza.
\newblock Strict dissipativity of {C}attaneo-{C}hristov systems for
  compressible fluid flow.
\newblock {\em J. Phys. A: Math. Theor.}, 53, (2020).

\bibitem{benzoserre}
S.~Benzoni-Gavage and D.~Serre.
\newblock {\em Multi-dimensional hyperbolic partial differential equations:
  First-order Systems and Applications}.
\newblock OUP Oxford, 2006.

\bibitem{christov}
C.~I. Christov.
\newblock On frame indifferent formulation of the{M}axwell-{C}attaneo model of
  finite-speed heat conduction.
\newblock {\em Mech. Research Comm}, 36:481--486, (2009).

\bibitem{jordy}
C.~I. Christov and P.~M. Jordan.
\newblock Heat {C}onduction {P}aradox {I}nvolving {S}econd-{S}ound
  {P}ropagation in {M}oving {Media}.
\newblock {\em Physical Review Letters}, 94:154301, (2005).

\bibitem{daf}
C.~Dafermos.
\newblock {\em Hyperbolic Conservation Laws in Continuum Physics}.
\newblock Springer, (Berlin 2016).

\bibitem{pseudofri}
K.~O. Friedrichs.
\newblock {\em Pseudo-differential operators}.
\newblock Lecture {N}otes in {C}ourant {I}nstitute, 1968-69.

\bibitem{kajineed}
K.~Kajitani.
\newblock A necessary condition for the $ l^{2}$-well posed {C}auchy problem
  with variable coefficients.
\newblock {\em Journal of Mathematics of Kyoto University}, 13(2):391--401,
  1973.

\bibitem{kajitani}
K.~Kajitani.
\newblock Strongly hyperbolic systems with variable coefficients.
\newblock {\em Publications of the Research Institute for Mathematical
  Sciences}, 9(3):597--612, 1974.

\bibitem{kano}
T.~Kano.
\newblock A {N}ecessary {C}ondition for the {W}ell-posedness of the {C}auchy
  {P}roblem for the {F}irst {O}rder {H}yperbolic {S}ystem with {M}ultiple
  {C}haracteristics.
\newblock {\em Publ. RIMS, Kyoto Univ.}, 5:149--164, (1969).

\bibitem{kasahara}
K.~Kasahara and M.~Yamaguti.
\newblock Strongly hyperbolic systems of linear partial differential equations
  with constant coefficients.
\newblock {\em Memoirs of the College of Science, University of Kyoto. Series
  A: Mathematics}, 33(1):1--23, 1960.

\bibitem{katosym}
T.~Kato.
\newblock The {C}auchy {P}roblem for {Q}uasi-linear {S}ymmetric {H}yperbolic
  {S}ystems.
\newblock {\em Arch. Rational Mech. Anal}, 58:181--205, (1975).

\bibitem{kawayong}
S.~Kawashima and W.~A. Yong.
\newblock Dissipative structure and entropy for hyperbolic systems of balance
  laws.
\newblock {\em Arch. {R}ation. {M}ech. {A}nal.}, 174(3):345--364, 2004.

\bibitem{otto}
H.~O. Kreiss and J.~Lorenz.
\newblock {\em Initial-{B}oundary {V}alue {P}roblems and the {N}avier-{S}tokes
  {E}quations}.
\newblock SIAM, (Philadelphia 2004).

\bibitem{laxsymm}
P.D. Lax.
\newblock Differential equations, difference equations and matrix theory.
\newblock {\em Commun. Pure Appl. Math.}, 11(2):175--194, 1958.

\bibitem{mascialini}
C.~Mascia and R.~Natalini.
\newblock On relaxation hyperbolic systems violating the {S}hizuta--{K}awashima
  condition.
\newblock {\em Arch. {R}ation. {M}ech. {A}nal.}, 195:729--762, 2010.

\bibitem{metivierl2}
G.~M{\'e}tivier.
\newblock $ l^{2}$ well-posed {C}auchy problems and symmetrizability of first
  order systems.
\newblock {\em Journal de l’Ecole polytechnique-Math{\'e}matiques}, 1:39--70,
  2014.

\bibitem{metivierpara}
Guy Metivier.
\newblock Para-differential calculus and applications to the cauchy problem for
  nonlinear systems.
\newblock 2008.

\bibitem{amorro}
A.~Morro.
\newblock A {T}hermodynamic {A}pproach to {R}ate {E}quations in {C}ontinuum
  {P}hysics.
\newblock {\em J. Phys. Sci. Appl.}, 7:15--23, (2017).

\bibitem{morro2}
A.~Morro.
\newblock Modelling of elastic heat conductors via objective rate equations.
\newblock {\em Contin. Mech. Thermodyn.}, 30:1231--1243, (2018).

\bibitem{serre}
D.~Serre.
\newblock {\em Conservation Laws 1: Hyperbolicity, Entropies, Shock Waves}.
\newblock Cambridge University Press, (Cambridge 2003).

\bibitem{kawazuta}
Y.~Shizuta and S.~Kawashima.
\newblock Systems of equations of hyperbolic-parabolic type with applications
  to the discrete boltzmann equation.
\newblock {\em Hokkaido Mathematical Journal}, 14(2):249--275, 1985.

\bibitem{gstran}
G.~Strang.
\newblock Necessary and {I}nsufficient {C}onditions for {W}ell-{P}osed {C}auchy
  {P}roblems.
\newblock {\em J. Differ. Equ.}, 2:107--114, (1966).

\bibitem{stra}
B.~Straughan.
\newblock Acoustic waves in a {C}attaneo-{C}hristov gas.
\newblock {\em Phys. Lett. A}, 374:2667--2669, (2010).

\bibitem{weyl}
H.~Weyl.
\newblock Shock waves in arbitrary fluids.
\newblock {\em Comm. Pure Appl. Math.}, 2:103--122, (1949).

\bibitem{zengas}
Y.~Zeng.
\newblock Gas dynamics in thermal nonequilibrium and general hyperbolic systems
  with relaxation.
\newblock {\em Arch. {R}ation. {M}ech. {A}nal.}, 150:225--279, 1999.

\end{thebibliography}
	
\end{document}